\documentclass[12pt]{article}

\usepackage{amsmath,amssymb,amsfonts,amsthm}
\usepackage{geometry}
\usepackage{cite}
\usepackage{graphicx}

\usepackage[colorlinks=true,
linkcolor=blue,
citecolor=blue,
urlcolor=blue]{hyperref}

\usepackage{titling}

\pretitle{%
\vspace{-3em}
\begin{center}
\rule{\linewidth}{2.5pt}\vspace{0.8em}
\LARGE\bfseries
}

\posttitle{%
\par\vspace{-0.2em}
\rule{\linewidth}{0.5pt}
\end{center}
\vspace{0.5em}
}

\numberwithin{equation}{section}

\newtheorem{theorem}{Theorem}[section]

\newtheorem{corollary}{Corollary}[section]

\theoremstyle{definition}

\newtheorem{remark}{Remark}[section]

\title{Two fixed functions can approximate any continuous function using only addition and composition}

\author{
\textbf{Vugar E. Ismailov}\thanks{
E-mail: \texttt{vugaris@gmail.com},
\texttt{vugaris@mail.ru}.}
}

\date{}

\begin{document}

\maketitle

\begin{abstract}
We prove that two fixed univariate functions, namely, an arbitrary
continuous non-affine function and a particular affine function, are
sufficient to approximate continuous functions of one variable under the
operations of addition and composition. The same fixed functions can
also be used to approximate multivariate continuous functions, provided
that the coordinate functions are also available. We also show that the
number of generators can be reduced from two to one. We construct a
specific continuous function that generates a dense class in the
univariate setting and, together with the coordinate functions, in the
multivariate setting.

\medskip


\medskip

\end{abstract}

\section{Introduction}

Superpositions of functions arise naturally in many branches of
mathematics and in numerous areas of applications. In many situations,
one attempts to represent or approximate complicated multivariate
functions by combining simpler functions through elementary operations.
Among such constructions, superpositions formed by addition and
composition are of particular importance.

Perhaps the first major attention to superpositions was drawn by
Hilbert's 13th problem. In his celebrated address at the International
Congress of Mathematicians held in Paris in 1900, Hilbert presented 23
fundamental mathematical problems intended to guide research in the
coming century. The 13th problem asks \emph{whether every continuous function
of three variables can be represented as a superposition of continuous
functions of two variables}. Hilbert believed that such a representation
should not exist, since he expected that certain functions possess a
genuinely multivariate character and therefore cannot be reduced to
superpositions of functions of fewer variables.

For more than fifty years, attempts to solve the problem were largely
directed toward confirming Hilbert's conjecture. However, in 1957, while still a teenager,
Arnold~\cite{Arnold1957} proved that every continuous function of three
variables can in fact be represented as a superposition of continuous
functions of two variables (see also~\cite{Arnold1959}). 
Shortly afterwards, his teacher Kolmogorov~\cite{Kolmogorov1957} obtained a stronger and more
elegant result, now known as the Kolmogorov--Arnold representation
theorem. It states that every continuous function
\[
f:[0,1]^n \to \mathbb{R}
\]
admits a representation of the form
\begin{equation}\label{eq:kart}
f(x_1,\dots,x_n)
=
\sum_{k=1}^{2n+1}
\Phi_k\!\left(
\sum_{j=1}^{n}\psi_{kj}(x_j)
\right),
\end{equation}
where $\psi_{kj}:\mathbb{R}\to\mathbb{R}$ are fixed continuous functions
and $\Phi_k:\mathbb{R}\to\mathbb{R}$ are continuous functions depending
on $f$. Thus the Kolmogorov--Arnold representation theorem shows that every
continuous multivariate function can be represented through \emph{sums and
compositions of univariate functions}. This theorem has been
revisited, refined, and extended in several directions (see, e.g.,
\cite[Chapter~1]{Khavinson1997},
\cite[Chapter~4]{Ismailov2021} and the references therein).

It should be noted that representation \eqref{eq:kart} remains
valid not only for continuous functions, but also for all discontinuous
functions $f$, with the same inner functions $\psi_{kj}$ as in the
continuous case and with outer functions $\Phi_k$ depending on $f$
(see \cite{Ismailov2023}). Moreover, the inner functions $\psi_{kj}(x_j)$ can
be chosen in the form $\lambda_j \psi(x_j+\epsilon k)$, and the outer
functions $\Phi_k$ can be replaced by a single function $\Phi$, both in
the continuous case \cite{Lorentz1962,Sprecher1965} and in the
discontinuous case \cite{Ismayilova2024}.

Despite its conceptual depth, the Kolmogorov--Arnold representation
theorem has important limitations from the viewpoint of applications.
The universal inner functions occurring in the theorem are rather
irregular~\cite{Vitushkin1967} and computationally 
impractical~\cite{Hieber2021}. Had these functions possessed sufficient
regularity and computability, multivariate approximation theory might have developed in a
very different direction.

The influence of Kolmogorov's theorem extended far beyond pure
mathematics. In particular, it played an important role in the
development of neural network theory. Its connection with feedforward
neural networks was first emphasized by Hecht--Nielsen~\cite{Hecht1987}. More recently,
the theorem has inspired the development of Kolmogorov--Arnold networks
(KANs), introduced as architectures based on compositions and sums of
univariate functions~\cite{Liu2024}. In KANs, nonlinear univariate functions are placed
on the edges of the network, while multivariate functions are generated
through repeated compositions of these univariate functions and
addition operations. The edge functions are typically parameterized by
splines.

A KAN computes functions of the form
\begin{equation*}
\begin{aligned}
f(x_1,\dots,x_n)
=
\sum_{k_L=1}^{N_L}
\phi_{k_L}
\Biggl(
\sum_{k_{L-1}=1}^{N_{L-1}}
\psi^{(L)}_{k_L k_{L-1}}
\Biggl(
\cdots
\sum_{k_1=1}^{N_1}
\psi^{(2)}_{k_2 k_1}
\Bigl(
\sum_{j=1}^{n}
\psi^{(1)}_{k_1 j}(x_j)
\Bigr)
\cdots
\Biggr)
\Biggr),
\end{aligned}
\end{equation*}
where all functions involved are univariate. This representation shows
that KANs generate multivariate functions through nested compositions of
sums of univariate functions.

Thus KANs are closely connected with the general philosophy of the
Kolmogorov--Arnold representation theorem: multivariate functions are
constructed from univariate functions through addition and
composition. However, even though KANs employ relatively simple and
practical functions such as splines, many different univariate functions
are typically involved in the construction, and their number increases
with the desired approximation accuracy. For approximation-theoretic results 
concerning KANs, see~\cite{Ismailov2026} and the references therein.

A related but different line of research is the theory of approximation
by feedforward neural networks. A basic notion in this area is
the universal approximation property. Roughly speaking, a class of
networks possesses this property if it can approximate every continuous
function on compact sets with arbitrary accuracy. The first results of
this type were obtained for shallow networks, that is, networks with a
single hidden layer. Such networks compute functions of the form
\begin{equation}\label{eq:shallow_network}
\sum_{i=1}^{r}
c_i\,
\sigma(\mathbf{a}_i\cdot\mathbf{x}+b_i),
\end{equation}
where $r\in\mathbb N$, $c_i,b_i\in\mathbb R$,
$\mathbf{a}_i\in\mathbb R^n$, and
$\sigma:\mathbb R\to\mathbb R$ is the activation function.

Approximation properties of shallow networks were extensively studied by
many authors, beginning with the works of 
Gallant and White~\cite{Gallant1988}, Cybenko~\cite{Cybenko1989},
Funahashi~\cite{Funahashi1989}, Hornik, Stinchcombe, and White~\cite{Hornik1989}, 
among others. A complete characterization was obtained in the influential 1993 paper
of Leshno, Lin, Pinkus, and Schocken~\cite{Leshno1993}: for a continuous activation
function, the family \eqref{eq:shallow_network} is dense in $C(K)$ for
every compact set $K\subset\mathbb R^n$ if and only if $\sigma$ is not a polynomial.

Modern applications often employ deep networks with several hidden
layers. A typical two-hidden-layer network computes functions of the
form
\[
\sum_{j=1}^{N_2}
v_j\,
\sigma
\left(
\sum_{i=1}^{N_1}
w_{ji}\,
\sigma(\mathbf{a}_i\cdot\mathbf{x}+b_i)
+d_j
\right),
\]
and deeper architectures are obtained by iterating this construction.
Thus deep neural networks may also be viewed as systems generated by
superpositions of a single univariate function together with many affine
functions. As the approximation error decreases, the number of affine
functions involved in the representation generally increases.

The preceding discussion reveals two different directions in the theory
of superpositions. On the one hand, the Kolmogorov--Arnold
representation theorem uses only finitely many fixed functions, but
these functions are rather irregular and not suitable for practical
implementation. On the other hand, neural network theory employs simple
and practically useful functions. For example, traditional feedforward
neural networks use superpositions of a single nonlinear activation
function with many affine functions, while Kolmogorov--Arnold networks
involve many simple nonlinear univariate functions, typically
spline-based edge functions. In both settings, however, the number of
functions involved in the representation generally increases as the
desired approximation accuracy improves.

This naturally leads to the following question.

\begin{center}
\emph{
Can one approximate arbitrary continuous functions using only a finite
family of simple fixed functions together with the operations of
addition and composition?
}
\end{center}

The purpose of the present paper is to show that the answer is
affirmative. We prove that two fixed functions are sufficient. More
precisely, we show that any fixed continuous non-affine function together with
one concrete affine function generate, under repeated addition and
composition, classes dense in spaces of continuous functions.

Both the univariate and multivariate settings are considered. In the
univariate case, approximation is obtained by repeated addition and
composition of two fixed functions
(Theorem~\ref{thm:two_function_univariate}). In the multivariate
case, one works with the coordinate functions
\[
x_1,\dots,x_n,
\]
together with the same two fixed functions and the operations of
addition and composition
(Theorem~\ref{thm:two_function_multivariate}).
Thus every continuous multivariate function can be approximated
arbitrarily well using only one continuous non-affine function, one
affine function, the coordinate functions, and the operations of
addition and composition.

It is important to emphasize that arbitrary multiplication by constants
is not assumed in advance. The necessary affine operations are generated
recursively from the fixed functions.

We also show that the number of generators can be reduced from two to
one. More precisely, we explicitly construct a continuous function that
generates a dense class in the univariate setting
(Corollary~\ref{cor:single_function_univariate}) and, together with the
coordinate functions, in the multivariate setting
(Corollary~\ref{cor:single_function_multivariate}).

The rest of the paper is organized as follows. In
Section~\ref{sec:approximation} we define the generated classes, prove
the approximation theorems, and discuss the recursive structure of the
generated functions. Section~\ref{sec:remarks} contains several remarks
and possible generalizations. In
Section~\ref{sec:single_function} we show that the number of generators
can be reduced from two to one by constructing a specific continuous
function that serves as a single generator.

\section{Approximation by superpositions of two functions}\label{sec:approximation}

In this section we prove that arbitrary continuous functions can be
approximated using only two fixed univariate functions together with the
operations of addition and composition. More precisely, we show that one
arbitrary continuous non-affine function together with one concrete
affine function generate classes dense in spaces of continuous
functions.

We first consider the univariate case and then pass to the
multivariate setting.

\subsection*{The univariate case}

The following theorem holds.

\begin{theorem}\label{thm:two_function_univariate}
Let $g_1:\mathbb R\to\mathbb R$ be continuous and non-affine, and let
\[
g_2(t)=1-\frac{t}{2}.
\]
Then, for every compact set $K\subset\mathbb R$, every function in
$C(K)$ can be uniformly approximated on $K$ by functions obtained from
$g_1$ and $g_2$ by repeated addition and composition.
\end{theorem}

\begin{proof}
Let $\mathcal S$ denote the class of all functions obtained from
$g_1$ and $g_2$ by repeated addition and composition. For example, the functions
\[
g_1(g_2(t))+g_2(g_1(t)),
\quad
g_1(g_1(g_2(t))),
\quad
g_2(g_2(t))+g_1(t),
\quad
g_1(g_2(g_1(t)+g_2(t)))
\]
belong to $\mathcal S$. It is important to note that multiplication 
by constants is not assumed a priori.

We first show that $\mathcal S$ contains the functions
\[
t\mapsto 1,
\qquad
t\mapsto t,
\qquad
t\mapsto -t,
\qquad
t\mapsto \frac t2.
\]

Since
\[
g_2(t)=1-\frac t2,
\]
we have
\[
g_2(g_2(t))
=
1-\frac12\left(1-\frac t2\right)
=
\frac12+\frac t4.
\]
Hence
\[
g_2(t)+2g_2(g_2(t))
=
\left(1-\frac t2\right)
+
2\left(\frac12+\frac t4\right)
=
2.
\]
Thus the constant function $2$ belongs to $\mathcal S$.

Applying $g_2$ twice to the constant function $2$, we obtain
\[
g_2(g_2(2))=1,
\]
hence the constant function $1$ belongs to $\mathcal S$.

Since $2\in\mathcal S$, we also have
\[
4=2+2\in\mathcal S.
\]
Therefore
\[
g_2(4)=1-\frac42=-1,
\]
so the constant function $-1$ belongs to $\mathcal S$.

Next,
\[
2g_2(g_2(t))-1
=
\left(1+\frac t2\right)-1
=
\frac t2,
\]
and hence
\[
t=\frac t2+\frac t2
\]
belongs to $\mathcal S$.

Finally,
\[
g_2(t)-1
=
-\frac t2,
\]
and therefore
\[
-t=-\frac t2-\frac t2
\]
belongs to $\mathcal S$.

Thus
\[
t\mapsto 1\in\mathcal S,\qquad
t\mapsto t\in\mathcal S,\qquad
t\mapsto -t\in\mathcal S,\qquad
t\mapsto \frac t2\in\mathcal S.
\]

Since $t/2\in\mathcal S$ and $\mathcal S$ is closed under composition,
it follows that if $f\in\mathcal S$, then $f/2\in\mathcal S$.
Repeating this argument, we obtain $f/2^r\in\mathcal S$ for every
$r\in\mathbb N$.
Since $\mathcal S$ is closed under addition and sign changes, it follows
that
\[
\frac{m}{2^r}f\in\mathcal S
\]
for every dyadic rational number $m/2^r$.

In particular, every affine function
\[
t\mapsto at+b,
\]
where $a$ and $b$ are dyadic rationals, belongs to $\mathcal S$.

We distinguish two cases: either $g_1$ is nonpolynomial or $g_1$ is a
polynomial of degree at least $2$.

Assume first that $g_1$ is nonpolynomial. By the theorem of
Leshno, Lin, Pinkus, and Schocken, since $g_1$ is
continuous and nonpolynomial, finite sums of the form
\[
\sum_{k=1}^N
c_k\,g_1(a_k t+b_k)
\]
are dense in $C(K)$ for every compact set $K\subset\mathbb R$.

Since dyadic rationals are dense in $\mathbb R$, the coefficients
$a_k,b_k,c_k$ may be approximated arbitrarily closely by dyadic
rationals. Since $g_1$ is continuous, it follows that finite sums of the
form
\[
\sum_{k=1}^N
\alpha_k\,g_1(\beta_k t+\gamma_k),
\]
where $\alpha_k,\beta_k,\gamma_k$ are dyadic rationals, are also dense
in $C(K)$.

Each affine function
\[
t\mapsto \beta_k t+\gamma_k
\]
belongs to $\mathcal S$, and therefore each function
\[
g_1(\beta_k t+\gamma_k)
\]
belongs to $\mathcal S$ by composition. Since $\mathcal S$ is closed
under dyadic scalar multiplication and addition, every finite sum of the
above form belongs to $\mathcal S$. Hence $\mathcal S$ is dense in
$C(K)$.

Now assume that $g_1$ is a polynomial of degree at least $2$. We use
finite differences to reduce the degree of $g_1$ and obtain a quadratic
polynomial.

Recall that for a function $f$ and a step size $h$, the finite
difference operator is defined by
\[
\Delta_h f(t)=f(t+h)-f(t).
\]
Its iterates are defined recursively by
\[
\Delta_h^m f
=
\Delta_h(\Delta_h^{m-1}f)
=
\sum_{r=0}^{m}
(-1)^{m-r}
\binom{m}{r}
f(t+rh).
\]
A basic property of finite differences states that if $f$ is a
polynomial of degree $d$, then $\Delta_h f$ is a polynomial of degree
$d-1$. Consequently, $\Delta_h^{d-2}f$ is a quadratic polynomial.

We apply this construction with $h=1$. If $d=\deg g_1$, then
\[
q(t)
=
\Delta_1^{d-2}g_1(t)
=
\sum_{r=0}^{d-2}
(-1)^{d-2-r}
\binom{d-2}{r}
g_1(t+r)
\]
is a quadratic polynomial.

Since every affine function with dyadic coefficients belongs to
$\mathcal S$, all shifts
\[
t\mapsto t+r
\]
belong to $\mathcal S$. Therefore, $q\in\mathcal S$.

Write
\[
q(t)=At^2+Bt+C,
\qquad A\ne0.
\]

Let $\overline{\mathcal S}$ denote the closure of $\mathcal S$ with
respect to uniform convergence on compact sets.

We claim that
\[
t^2\in\overline{\mathcal S}.
\]
Indeed, since
\[
t^2=\frac1Aq(t)-\frac BA\,t-\frac CA,
\]
and the dyadic rationals are dense in $\mathbb R$, for every compact set
$K\subset\mathbb R$ and every $\varepsilon>0$ there exist dyadic numbers
$\alpha,\beta,\gamma$ such that
\[
\left\|
\alpha q(t)+\beta t+\gamma-t^2
\right\|_{K}
<
\varepsilon.
\]
Since $q,t,1\in\mathcal S$ and $\alpha,\beta,\gamma$ are dyadic,
\[
\alpha q(t)+\beta t+\gamma\in\mathcal S.
\]
Thus $t^2\in\overline{\mathcal S}$.

We next show that $\overline{\mathcal S}$ is closed under addition,
composition, and multiplication.

Closure under addition follows immediately from the definition of
$\overline{\mathcal S}$.
To prove closure under composition, let $u,v\in\overline{\mathcal S}$.
Choose sequences $u_n,v_n\in\mathcal S$ such that
\[
u_n\to u,
\qquad
v_n\to v
\]
uniformly on compact subsets of $\mathbb R$.

Let $K\subset\mathbb R$ be an arbitrary compact set. Since $v_n\to v$ 
uniformly on $K$, there exists a compact set $L\subset\mathbb R$
containing both $v(K)$ and $v_n(K)$ for all sufficiently large $n$.
Indeed, since $v(K)$ is compact, the set
\[
L:=\{y\in\mathbb R:\operatorname{dist}(y,v(K))\le1\}
\]
is compact and contains $v(K)$. Moreover, since $v_n\to v$ uniformly on $K$,
there exists $N$ such that
\[
|v_n(x)-v(x)|\le1
\]
for all $x\in K$ and all $n\ge N$. Hence $v_n(K)\subset L$
for all $n\ge N$.

Now let us write
\[
u_n(v_n(x))-u(v(x))
=
\bigl(u_n(v_n(x))-u(v_n(x))\bigr)
+
\bigl(u(v_n(x))-u(v(x))\bigr).
\]
Since $u_n\to u$ uniformly on $L$, the first term converges uniformly
to $0$ on $K$. Moreover, since $u$ is uniformly continuous on $L$ and
$v_n\to v$ uniformly on $K$, the second term also converges uniformly
to $0$ on $K$. Therefore
\[
u_n\circ v_n\to u\circ v
\]
uniformly on $K$. Since $u_n\circ v_n\in\mathcal S$, it follows that
$u\circ v\in\overline{\mathcal S}$.

Now let $u,v\in\overline{\mathcal S}$. Since $\overline{\mathcal S}$ is
closed under addition and composition and since $t^2\in\overline{\mathcal S}$,
we have
\[
(u+v)^2\in\overline{\mathcal S},
\qquad
u^2\in\overline{\mathcal S},
\qquad
v^2\in\overline{\mathcal S}.
\]
Hence
\[
(u+v)^2-u^2-v^2=2uv
\]
belongs to $\overline{\mathcal S}$. Since $t\mapsto t/2\in\mathcal S$,
composing with this function yields
\[
uv\in\overline{\mathcal S}.
\]
Therefore $\overline{\mathcal S}$ is closed under multiplication.

Since $1,t\in\overline{\mathcal S}$ and $\overline{\mathcal S}$ is closed
under addition and multiplication, every polynomial with dyadic
coefficients belongs to $\overline{\mathcal S}$. Since dyadic 
rationals are dense in $\mathbb R$, every polynomial can be
uniformly approximated on compact sets by polynomials with dyadic
coefficients. Therefore every polynomial belongs to
$\overline{\mathcal S}$. By the Weierstrass approximation theorem,
polynomials are dense in $C(K)$ for every compact set $K\subset\mathbb R$.
Hence $\mathcal S$ is dense in $C(K)$.
\end{proof}

\subsection*{The multivariate case}

The previous theorem concerns approximation of univariate functions.
For multivariate functions, one additionally allows the coordinate
functions and considers the class generated from them together with the
same two fixed functions under repeated addition and composition.

\begin{theorem}\label{thm:two_function_multivariate}
Let $g_1:\mathbb R\to\mathbb R$ be continuous and non-affine, and define
\[
g_2(t)=1-\frac{t}{2}.
\]

Let $\mathcal S_n$ be the smallest class of real-valued functions on
$\mathbb R^n$ satisfying the following properties:
\begin{enumerate}
\item[(i)]
the coordinate functions
\[
x_1,\dots,x_n
\]
belong to $\mathcal S_n$;

\item[(ii)]
if $u,v\in\mathcal S_n$, then
\[
u+v\in\mathcal S_n;
\]

\item[(iii)]
if $u\in\mathcal S_n$, then
\[
g_i\circ u\in\mathcal S_n,
\qquad i=1,2.
\]
\end{enumerate}

Then, for every compact set $K\subset\mathbb R^n$, the class
$\mathcal S_n$ is dense in $C(K)$.
\end{theorem}

\begin{proof}
As in the proof of
Theorem~\ref{thm:two_function_univariate}, the univariate functions
\[
1,
\qquad
t,
\qquad
-t,
\qquad
\frac t2
\]
can be generated from $g_2$ by repeated addition and composition.
Consequently, if $u\in\mathcal S_n$, then composition of $u$ with these
univariate functions yields
\[
1\in\mathcal S_n,
\qquad
-u\in\mathcal S_n,
\qquad
\frac u2\in\mathcal S_n.
\]
By repeated addition and repeated composition with the function
$t\mapsto t/2$, one obtains
\[
\frac{m}{2^r}u\in\mathcal S_n
\]
for all $m\in\mathbb Z$ and $r\in\mathbb N\cup\{0\}$.

Since the constant function $1$ belongs to $\mathcal S_n$, all
dyadic constant functions also belong to $\mathcal S_n$. Hence, for
every $u\in\mathcal S_n$ and every $a,b\in\mathbb D$,
\[
au+b\in\mathcal S_n,
\]
where $\mathbb D$ denotes the set of dyadic rational numbers.
In particular, since $x_1,\dots,x_n\in\mathcal S_n$, every function of
the form
\[
\mathbf x\mapsto \alpha_1x_1+\cdots+\alpha_nx_n+\beta,
\qquad
\alpha_1,\dots,\alpha_n,\beta\in\mathbb D,
\]
belongs to $\mathcal S_n$.

Assume first that $g_1$ is nonpolynomial. By the theorem of Leshno,
Lin, Pinkus, and Schocken, finite sums of the form
\[
\sum_{k=1}^N
c_k\,g_1(\mathbf a_k\cdot \mathbf x+b_k),
\]
where $c_k,b_k\in\mathbb R$ and $\mathbf a_k\in\mathbb R^n$,
are dense in $C(K)$ for every compact set
$K\subset\mathbb R^n$.

Approximating the parameters by dyadic numbers and using continuity
of $g_1$, we conclude that finite sums of the form
\[
\sum_{k=1}^N
\gamma_k\,
g_1(\boldsymbol\alpha_k\cdot \mathbf x+\beta_k),
\]
where $\gamma_k,\beta_k$ belong to the set $\mathbb D$ of dyadic
rational numbers and $\boldsymbol\alpha_k\in\mathbb D^n$,
are also dense in $C(K)$. 
For completeness, let us briefly justify this assertion.

Fix
\[
H(\mathbf x)
=
\sum_{k=1}^N
c_k\,g_1(\mathbf a_k\cdot \mathbf x+b_k)
\]
and let $\varepsilon>0$ be arbitrary. Since there are finitely many
terms, it is enough to approximate each term within $\varepsilon/N$.

Consider one term
\[
c_k\,g_1(\mathbf a_k\cdot \mathbf x+b_k).
\]
Since $K$ is compact, the set
\[
\{\mathbf a_k\cdot \mathbf x+b_k:\mathbf x\in K\}
\]
is compact in $\mathbb R$. Choose a compact interval
\[
I=[m-1,M+1]
\]
containing this set, where
\[
m=\min_{\mathbf x\in K}(\mathbf a_k\cdot \mathbf x+b_k),
\qquad
M=\max_{\mathbf x\in K}(\mathbf a_k\cdot \mathbf x+b_k).
\]
Since $g_1$ is continuous, it is uniformly continuous on $I$.
Therefore, for every $\eta>0$ there exists $\delta\in(0,1)$ such that
\[
|u-v|<\delta
\quad\Longrightarrow\quad
|g_1(u)-g_1(v)|<\eta
\]
for all $u,v\in I$.

Now choose dyadic numbers $\gamma_k,\beta_k$ and a vector
\[
\boldsymbol\alpha_k
=
(\alpha_{k1},\dots,\alpha_{kn})
\in\mathbb D^n
\]
such that
\[
|\gamma_k-c_k|<\eta,
\qquad
|b_k-\beta_k|<\frac{\delta}{2},
\]
and
\[
|a_{kj}-\alpha_{kj}|
<
\frac{\delta}{2nL},
\qquad j=1,\dots,n,
\]
where
\[
L=\max_{\mathbf x\in K}\max_{1\le j\le n}|x_j|.
\]

Then for every $\mathbf x\in K$,
\[
\begin{aligned}
\left|
(\mathbf a_k\cdot \mathbf x+b_k)
-
(\boldsymbol\alpha_k\cdot \mathbf x+\beta_k)
\right|
&\le
\sum_{j=1}^n
|a_{kj}-\alpha_{kj}||x_j|
+
|b_k-\beta_k| \\
&<
n\cdot\frac{\delta}{2nL}\cdot L
+
\frac{\delta}{2}
=
\delta.
\end{aligned}
\]
Since $\delta<1$, both
\[
\mathbf a_k\cdot \mathbf x+b_k
\quad\text{and}\quad
\boldsymbol\alpha_k\cdot \mathbf x+\beta_k
\]
belong to $I$. Hence
\[
\left|
g_1(\mathbf a_k\cdot \mathbf x+b_k)
-
g_1(\boldsymbol\alpha_k\cdot \mathbf x+\beta_k)
\right|
<
\eta
\]
for all $\mathbf x\in K$.

Choosing $\eta$ sufficiently small and hence $\gamma_k$
sufficiently close to $c_k$, we obtain
\[
\left|
c_k g_1(\mathbf a_k\cdot \mathbf x+b_k)
-
\gamma_k g_1(\boldsymbol\alpha_k\cdot \mathbf x+\beta_k)
\right|
<
\frac{\varepsilon}{N}
\]
for all $\mathbf x\in K$.

Summing over $k=1,\dots,N$, we conclude that $H$ can be uniformly
approximated on $K$ by finite sums with dyadic coefficients.

Since the coordinate functions belong to $\mathcal S_n$, every affine
function
\[
\mathbf x\mapsto \boldsymbol\alpha_k\cdot \mathbf x+\beta_k
\]
with dyadic coefficients belongs to $\mathcal S_n$. Hence every function
\[
g_1(\boldsymbol\alpha_k\cdot \mathbf x+\beta_k)
\]
belongs to $\mathcal S_n$. Therefore all finite sums of the form
\[
\sum_{k=1}^N
\gamma_k\,
g_1(\boldsymbol\alpha_k\cdot \mathbf x+\beta_k),
\]
with dyadic coefficients belong to $\mathcal S_n$. Since such sums are
dense in $C(K)$, the class $\mathcal S_n$ is also dense in $C(K)$.

Now consider the case where $g_1$ is a polynomial of degree at least $2$.
As in the proof of Theorem~\ref{thm:two_function_univariate}, finite
differences allow us to obtain $t^2$ in the closure of the univariate
class generated by $g_1$ and $g_2$. Moreover, the same argument as in
the proof of Theorem~\ref{thm:two_function_univariate} shows that if
$h$ belongs to the closure of the univariate generated class and
$u\in\overline{\mathcal S_n}$, then
\[
h\circ u\in\overline{\mathcal S_n}.
\]
Hence, if $u\in\overline{\mathcal S_n}$, then
\[
u^2\in\overline{\mathcal S_n}.
\]

Since $\overline{\mathcal S_n}$ is closed under addition and
multiplication by dyadic scalars, the polarization identity
\[
uv=\frac{(u+v)^2-u^2-v^2}{2}
\]
shows that it is also closed under multiplication. Since the
coordinate functions $x_1,\dots,x_n$ belong to $\mathcal S_n$,
every polynomial in $x_1,\dots,x_n$ with dyadic coefficients belongs
to $\overline{\mathcal S_n}$.

Since dyadic rationals are dense in $\mathbb R$, every polynomial can be
uniformly approximated on compact sets by polynomials with dyadic
coefficients. Therefore every polynomial in $x_1,\dots,x_n$ belongs to
$\overline{\mathcal S_n}$.
By the Stone--Weierstrass theorem, polynomials are dense in $C(K)$.
Therefore $\mathcal S_n$ is dense in $C(K)$.
\end{proof}

\subsection*{Structure of the generated classes}

In this subsection, we describe a recursive structure underlying the
class of univariate functions generated from $g_1$ and $g_2$ by repeated
addition and composition. The multivariate case can be treated in a
similar way.

Define classes $\mathcal G_m$ recursively by
\[
\mathcal G_0=\{g_1,g_2\},
\]
and
\[
\mathcal G_{m+1}
=
\mathcal G_m
\cup
\{f+h:\ f,h\in\mathcal G_m\}
\cup
\{f\circ h:\ f,h\in\mathcal G_m\},
\qquad m\ge0.
\]
Then the class of all functions obtained from $g_1$ and $g_2$ by finite
repeated applications of addition and composition is
\[
\mathcal G=\bigcup_{m=0}^{\infty}\mathcal G_m .
\]

Indeed, every function belonging to $\mathcal G$ is clearly obtained
from $g_1$ and $g_2$ by finitely many such operations. Conversely,
consider any finite expression $E$ formed from $g_1$ and $g_2$ using
only the operations $+$ and $\circ$. 
Let $L(E)$ denote the \emph{length} of $E$, that is, the number of
operation symbols $+$ and $\circ$ appearing in $E$. We prove by
induction on $L(E)$ that $E\in\mathcal G$.

If $L(E)=0$, then $E$ is either $g_1$ or $g_2$, and therefore
\[
E\in\mathcal G_0\subset\mathcal G.
\]

Assume now that every expression of length at most $r$ belongs to
$\mathcal G$, and let $E$ be an expression with $L(E)=r+1$.
Then the outermost operation in $E$ is either
addition or composition. Assume first that the outermost operation is
addition, so that
\[
E=F+H.
\]
Since
\[
L(E)=L(F)+L(H)+1,
\]
we have
\[
L(F)\le r,
\qquad
L(H)\le r.
\]
By the induction hypothesis, there exist indices $p,q\ge0$ such that
\[
F\in\mathcal G_p,
\qquad
H\in\mathcal G_q.
\]
Let
\[
N=\max\{p,q\}.
\]
Since
\[
\mathcal G_0\subseteq\mathcal G_1\subseteq\mathcal G_2\subseteq\cdots,
\]
we obtain
\[
F,H\in\mathcal G_N.
\]
Hence, by the definition of $\mathcal G_{N+1}$,
\[
F+H\in\mathcal G_{N+1}.
\]
Thus $E\in\mathcal G$.

The case where the outermost operation is composition is treated in the
same way. Therefore every finite expression generated from $g_1$ and
$g_2$ by repeated addition and composition belongs to $\mathcal G$.

For example,
\[
\mathcal G_1
=
\{
g_1,\,
g_2,\,
g_1+g_1,\,
g_1+g_2,\,
g_2+g_2,\,
g_1\circ g_1,\,
g_1\circ g_2,\,
g_2\circ g_1,\,
g_2\circ g_2
\},
\]
while $\mathcal G_2$ already contains more complicated expressions such as
\[
(g_1+g_2)\circ g_1,
\qquad
(g_1\circ g_2)+(g_2\circ g_1),
\qquad
(g_1\circ g_1)\circ g_2.
\]

It is also useful to observe how the expression length grows in this
construction. Since
\[
L(F+H)=L(F)+L(H)+1
\]
and similarly
\[
L(F\circ H)=L(F)+L(H)+1,
\]
it follows that if all expressions in $\mathcal G_{m-1}$ have length at
most $\ell$, then all expressions in $\mathcal G_m$ have length at most
$2\ell+1$.

Therefore expressions in
\[
\mathcal G_0,\mathcal G_1,\mathcal G_2,\mathcal G_3,\dots
\]
have lengths at most
\[
0,1,3,7,\dots
\]
respectively. In general, every expression in $\mathcal G_m$ has length
at most $2^m-1$.

Figure~\ref{fig:structure} illustrates the recursive construction of the
classes $\mathcal G_m$.

\begin{figure}[ht]
\centering
\includegraphics[width=\textwidth]{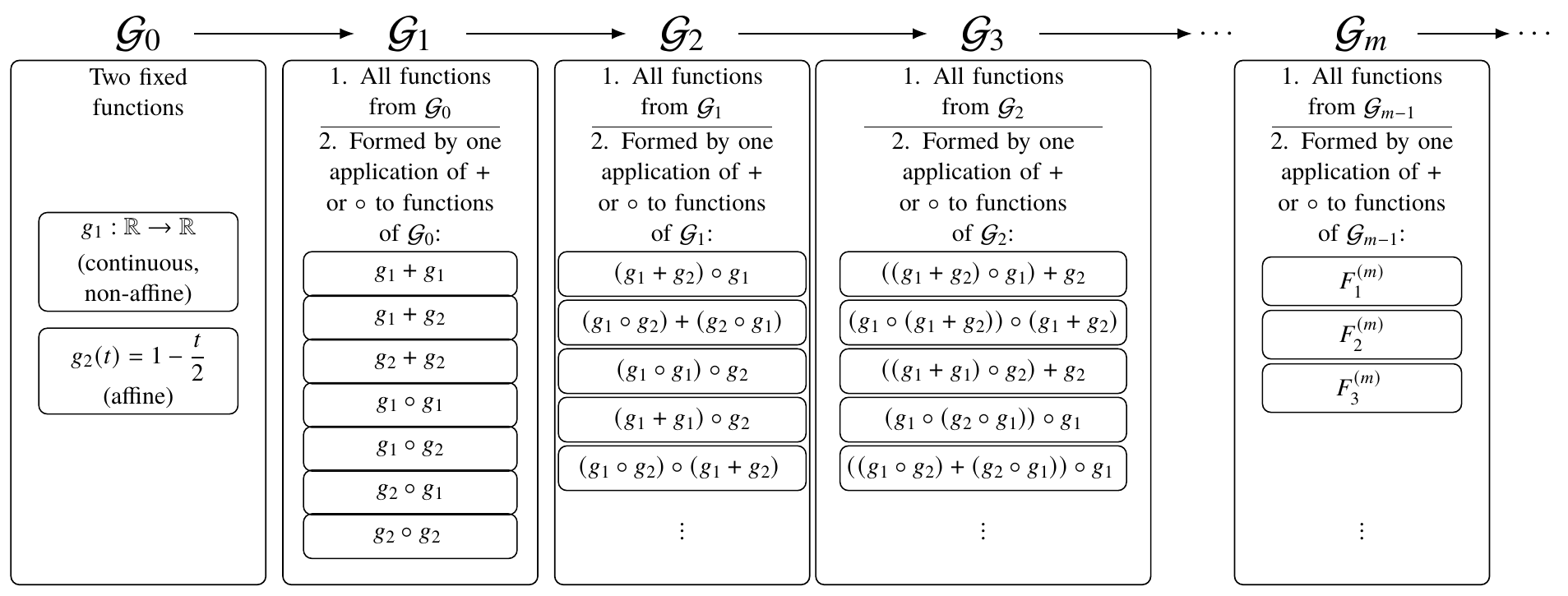}
\caption{Recursive construction of the classes $\mathcal G_m$ generated
from $g_1$ and $g_2$ using only addition and composition. For any
compact set $K\subset\mathbb R$, any $f\in C(K)$, and any $\varepsilon>0$,
there exist $m\in\mathbb N$ and $F\in\mathcal G_m$ such that
$\sup_{x\in K}|f(x)-F(x)|<\varepsilon$.}
\label{fig:structure}
\end{figure}

The multivariate case admits a similar interpretation. One
starts with
\[
\mathcal G_0=\{x_1,\dots,x_n,g_1,g_2\},
\]
and then recursively generates new classes by repeated addition and
composition.

\section{Some remarks}\label{sec:remarks}

In this section we collect several remarks concerning some aspects
of the proofs and possible generalizations of the affine function $g_2$.

\begin{remark}\label{rem:proof_mechanism}
Assume that we are given two univariate functions $g_1$ and $g_2$, and let
$\mathcal S$ denote the class of functions generated from $g_1$ and
$g_2$ by repeated addition and composition.

The proofs of
Theorems~\ref{thm:two_function_univariate}
and~\ref{thm:two_function_multivariate}
are based, in particular, on the following four properties:
\begin{enumerate}
\item [(1)] there exists a dense set $D\subset\mathbb R$ such that every
constant function $t\mapsto \lambda$, $\lambda\in D$, belongs to
$\mathcal S$;

\item [(2)] every affine map $t\mapsto at+b$, $a,b\in D$, belongs to
$\mathcal S$;

\item [(3)] for every $f\in\mathcal S$ and every $\lambda\in D$, the function
$\lambda f$ belongs to $\mathcal S$;

\item [(4)] additive inverses are available in $\mathcal S$.
\end{enumerate}

For the specific choice
\[
g_2(t)=1-\frac t2,
\]
these properties are obtained as follows. First, the functions
\[
t\mapsto 1,\qquad t\mapsto t,\qquad t\mapsto -t,\qquad
t\mapsto \frac t2
\]
belong to $\mathcal S$.

The constant function $1$, together with $-t$ and repeated composition
with $t/2$, allows one to generate all dyadic rational constants.
Thus property~(1) holds with
\[
D=
\left\{
\frac{m}{2^r}: m\in\mathbb Z,\ r\in\mathbb N\cup\{0\}
\right\}.
\]

Repeated composition with $t/2$ gives functions of the form
$t\mapsto t/{2^r}$. Since addition and additive inverses are available,
every function $t\mapsto {mt}/{2^r}$ belongs to $\mathcal S$.
Combining these functions with dyadic rational constants gives all
affine maps $t\mapsto at+b$, $a,b\in D$. Thus property~(2) holds.

Compositions of the functions $t\mapsto {mt}/{2^r}$ with functions
$f\in\mathcal S$ also give scalar multiplication by elements of $D$:
if $f\in\mathcal S$ and $\lambda\in D$, then $\lambda f\in\mathcal S$.
Thus property~(3) holds.

Finally, the function $-t$ gives additive inverses, so property~(4)
holds. This property is used, for instance, in the construction of
finite differences and in the identity
\[
(u+v)^2-u^2-v^2=2uv.
\]

Thus the present method depends not only on the choice of a non-affine
function $g_1$, but also on the fact that the affine function $g_2$
satisfies the four properties above. Any attempt to replace $g_2$ by a
more general affine function
\[
g_2(t)=at+b
\]
and prove the approximation results by the same method must take these
properties into account.
\end{remark}

\begin{remark}
The previous remark indicates that the proofs depend on the
satisfaction of the four properties listed there and for the specific choice
\[
g_2(t)=1-\frac t2,
\]
these properties are obtained through dyadic rational coefficients.

More generally, one may replace $g_2$ by
\[
g_2(t)=1-\frac tq,
\]
where $q\ge2$ is a fixed integer. In this case the same construction
produces coefficients from the set
\[
\mathbb D_q
=
\left\{
\frac{m}{q^r}
:
m\in\mathbb Z,\ r\in\mathbb N\cup\{0\}
\right\},
\]
which is dense in $\mathbb R$. Consequently, affine maps with
$q$-adic rational coefficients can be generated.

The proofs then proceed in the same way, except that the multiplication
step requires a minor modification. Indeed, when $q\ne2$, the function
$t\mapsto t/2$ need not belong to the generated class. However,
the multiplication step can still be recovered by approximation.
Since $\mathbb D_q$ is dense in $\mathbb R$, there exists a sequence
$\lambda_n\in\mathbb D_q$ such that $\lambda_n\to 1/2$.
Hence
\[
\lambda_n\bigl((u+v)^2-u^2-v^2\bigr)
\longrightarrow
uv
\]
uniformly on compact sets. Therefore the same approximation results
remain valid for every integer $q\ge2$.

One may also replace $g_2$ by
\[
g_2(t)=b-\frac tq,
\]
where $q\ge2$ is a fixed integer and $b\ne0$. Indeed,
\[
g_2(g_2(t))
=
b-\frac1q\left(b-\frac tq\right)
=
b\left(1-\frac1q\right)+\frac{t}{q^2}.
\]
Hence
\[
g_2(t)+q\,g_2(g_2(t))=qb.
\]
Thus the constant function $qb$ belongs to the generated class. Applying
$g_2$ to this constant function gives $g_2(qb)=0$,
so the zero function belongs to the generated class. Therefore $g_2(0)=b$
also belongs to the generated class.

Moreover, since the constant function $2qb$ can be obtained by repeated
addition, we have
\[
g_2(2qb)=b-\frac{2qb}{q}=-b.
\]
Thus both $b$ and $-b$ belong to the generated class.

Now
\[
g_2(t)+(-b)=-\frac tq,
\]
so $-t/q$ belongs to the generated class. Adding this function to itself
$q$ times gives $-t$. Hence sign changes are available.

Finally,
\[
g_2(-t)=b+\frac tq,
\]
and adding the constant function $-b$ gives $t/q$. Therefore $t/q$
belongs to the generated class, and adding $t/q$ to itself $q$ times
gives $t$.

Thus the functions
\[
b,\qquad -b,\qquad t,\qquad -t,\qquad \frac tq
\]
are generated using only addition and composition. Consequently,
affine maps
\[
t\mapsto at+c,
\qquad
a\in\mathbb D_q,
\quad
c\in b\mathbb D_q,
\]
can be generated, where
\[
b\mathbb D_q
=
\left\{
b\frac{m}{q^r}
:
m\in\mathbb Z,\ r\in\mathbb N\cup\{0\}
\right\}.
\]
Since both $\mathbb D_q$ and $b\mathbb D_q$ are dense in $\mathbb R$,
the approximation argument remains unchanged.
\end{remark}

\begin{remark}
The condition that $g_1$ be non-affine is best possible. Indeed, if
$g_1$ is affine, then every function obtained from $g_1$ and $g_2$ by
repeated addition and composition is again affine. Hence one cannot
approximate arbitrary continuous functions.

It is natural to ask whether the affine function $g_2$ can be replaced
by a more general affine map
\[
g_2(t)=at+b.
\]
Some restrictions are clearly necessary. If $a=0$, then $g_2$ is
constant and cannot produce dependence on the variable. If $a=\pm1$,
then repeated compositions generate only affine maps of the form
\[
t\mapsto \pm t+c,
\]
so the affine part of the construction becomes too restricted for the
above argument.

The proofs of
Theorems~\ref{thm:two_function_univariate}
and~\ref{thm:two_function_multivariate}
depend essentially on the ability to generate sufficiently many
constants, affine maps, scalar multiplications, and additive inverses
using only addition and composition.

In the special case
\[
g_2(t)=1-\frac t2,
\]
these properties can be verified explicitly (see Remark~\ref{rem:proof_mechanism}).

For a general affine map
\[
g_2(t)=at+b,
\qquad a\neq0,\pm1,
\]
it is not clear whether these algebraic properties remain available.
A complete characterization of affine maps for which the conclusions of
Theorems~\ref{thm:two_function_univariate}
and~\ref{thm:two_function_multivariate}
remain valid is beyond the scope of the methods discussed herein.
\end{remark}

\section{Is one function enough?}\label{sec:single_function}

Theorems~\ref{thm:two_function_univariate}
and~\ref{thm:two_function_multivariate}
show that dense classes can be generated from any continuous non-affine
function together with a suitable affine function. Thus, these results
apply to a large family of generating pairs.

A natural question is whether the number of generators can be reduced to
one. In this section we show that the answer to this question is affirmative.
We construct a specific continuous non-affine function
$g:\mathbb R\to\mathbb R$ that serves as a single generator in both the
univariate and multivariate settings.

For a function $f:\mathbb R\to\mathbb R$ and a positive integer $k$, we
write
\[
f^k=\underbrace{f\circ\cdots\circ f}_{k\ \text{times}}
\]
for the $k$-fold iterate of $f$.

The following theorem is the key ingredient.

\begin{theorem}\label{thm:single_function}
One can construct a continuous non-affine function
$g:\mathbb R\to\mathbb R$
such that
\[
g^{3}(t)=1-\frac{t}{2},
\qquad t\in\mathbb R.
\]
\end{theorem}

\begin{proof}
The affine map $t\mapsto 1-t/2$
has the unique fixed point $t=2/3$. We will construct a function
$g$ that also fixes this point.
We first construct a non-affine homeomorphism $\phi(u)$ of $\mathbb R$
whose third iterate is the translation $u\mapsto u-\ln 2$. The
function $g(t)$ will then be obtained from $\phi$ using the change of
variables $u=\ln|t-2/3|$.

Let
\[
a:=\ln 2.
\]
Choose a nonzero number $\varepsilon$ satisfying
\[
|\varepsilon|<\frac{a}{2\pi},
\]
and define
\[
\psi(u):=u+\varepsilon\sin\!\left(\frac{2\pi u}{a}\right).
\]
Since
\[
\psi'(u)
=
1+\frac{2\pi\varepsilon}{a}
\cos\!\left(\frac{2\pi u}{a}\right)
>0,
\]
the function $\psi$ is a strictly increasing homeomorphism of
$\mathbb R$.

Because the sine term is $a$-periodic,
\[
\psi(u-a)
=
u-a
+\varepsilon
\sin\!\left(
\frac{2\pi(u-a)}{a}
\right)
=
\psi(u)-a.
\]

Define
\[
\phi(u)
:=
\psi^{-1}\!\left(\psi(u)-\frac a3\right).
\]
Note that $\phi$ is also a strictly increasing homeomorphism of
$\mathbb R$, being the composition of two strictly increasing
homeomorphisms.

Since
\[
\psi(\phi(u))
=
\psi(u)-\frac a3,
\]
applying $\phi$ two more times gives
\[
\psi(\phi^3(u))
=
\psi(u)-a.
\]
Hence
\begin{equation}\label{eq:phi-cube}
\phi^3(u)=u-a.
\end{equation}

Since $\psi(u)-u$ is bounded, so is
$\psi^{-1}(v)-v$. Consequently,
\[
\phi(u)
=
\psi^{-1}\!\left(\psi(u)-\frac a3\right)
=
\psi(u)-\frac a3+O(1)
=
u-\frac a3+O(1).
\]
Hence
\[
\phi(u)\to-\infty
\qquad (u\to-\infty).
\]

Finally, $\phi$ is non-affine. Indeed, if $\phi$ were affine, then
\eqref{eq:phi-cube} would imply
\[
\phi(u)=u-\frac a3.
\]
Applying $\psi$ to the identity
\[
\phi(u)=\psi^{-1}\!\left(\psi(u)-\frac a3\right)
\]
would then yield
\[
\psi\!\left(u-\frac a3\right)=\psi(u)-\frac a3,
\]
which is false by the definition of $\psi$. Therefore $\phi$ is
non-affine.

Next define $g:\mathbb R\to\mathbb R$ by
\[
g(t)
:=
\begin{cases}
\dfrac23,
&
t=\dfrac23,
\\[2mm]
\dfrac23
-
\operatorname{sgn}\!\left(t-\dfrac23\right)
\exp\!\left(
\phi\!\left(
\ln\left|t-\dfrac23\right|
\right)
\right),
&
t\neq\dfrac23.
\end{cases}
\]

Since $\phi$ is continuous, $g$ is continuous on
$\mathbb R\setminus\{\frac23\}$. Moreover, since
\[
\phi(u)\to-\infty
\qquad (u\to-\infty),
\]
we have
\[
g(t)\to\frac23
\qquad
\left(t\to\frac23\right).
\]
Therefore $g$ is continuous on $\mathbb R$.

Let us show that $g$ is non-affine. If $g$ were affine, then, since
$g(2/3)=2/3$ and $g$ maps $(2/3,\infty)$ onto $(-\infty,2/3)$,
it would have the form
\[
g(t)=\frac23-\lambda\left(t-\frac23\right),
\qquad \lambda>0.
\]
For $t>2/3$, putting $u=\ln(t-2/3)$, we obtain
\[
\exp(\phi(u))=\lambda e^u,
\]
and hence
\[
\phi(u)=u+\ln\lambda .
\]
Thus $\phi$ would be affine, a contradiction. 
Therefore $g$ is non-affine.

Note that the factor $-\operatorname{sgn}(t-2/3)$ in the definition of
$g$ plays a crucial role in the above argument. It ensures that
\[
g\bigl((2/3,\infty)\bigr)=(-\infty,2/3)
\]
and
\[
g\bigl((-\infty,2/3)\bigr)=(2/3,\infty).
\]
Thus $g$ exchanges the two sides of the fixed point $2/3$.
Recall that the affine map $t\mapsto 1-t/2$ has the same property.

For $t=2/3$, we have
\[
g^{3}\!\left(\frac23\right)
=
\frac23
=
\left(1-\frac t2\right)\Big|_{t=2/3}.
\]

Now let $t\neq 2/3$, and set
\[
y:=t-\frac23.
\]
Set
\[
F(y)
:=
-\operatorname{sgn}(y)
\exp\!\bigl(\phi(\ln|y|)\bigr),
\qquad y\neq0.
\]
Then
\begin{equation}\label{eq:gF}
g(t)-\frac23
=
F\!\left(t-\frac23\right).
\end{equation}
Iterating \eqref{eq:gF} three times yields
\begin{equation}\label{eq:g3F3}
g^3(t)-\frac23
=
F^3\!\left(t-\frac23\right).
\end{equation}

Since
\[
\operatorname{sgn}(F(y))
=
-\operatorname{sgn}(y)
\]
and
\[
\ln|F(y)|
=
\phi(\ln|y|),
\]
an easy induction gives
\[
F^{k}(y)
=
(-1)^k\operatorname{sgn}(y)
\exp\!\bigl(\phi^{k}(\ln|y|)\bigr),
\qquad k\in\mathbb N.
\]
In particular,
\begin{equation}\label{eq:F3}
F^{3}(y)
=
-\operatorname{sgn}(y)
\exp\!\bigl(\phi^{3}(\ln|y|)\bigr).
\end{equation}
Using \eqref{eq:phi-cube} in \eqref{eq:F3}, we obtain
\[
F^{3}(y)
=
-\operatorname{sgn}(y)
\exp(\ln|y|-a).
\]
Since $a=\ln 2$,
\[
F^{3}(y)
=
-\frac12
\operatorname{sgn}(y)|y|
=
-\frac y2.
\]
Using this identity in \eqref{eq:g3F3}, we obtain
\[
g^{3}(t)-\frac23
=
-\frac12\left(t-\frac23\right).
\]
Hence
\[
g^{3}(t)
=
\frac23-\frac12\left(t-\frac23\right)
=
1-\frac t2.
\]

Thus
\[
g^{3}(t)=1-\frac t2
\]
for every $t\in\mathbb R$.
\end{proof}

Combining Theorem~\ref{thm:single_function} with
Theorems~\ref{thm:two_function_univariate}
and~\ref{thm:two_function_multivariate} yields the following
corollaries.

\begin{corollary}\label{cor:single_function_univariate}
There exists a continuous non-affine function
$g:\mathbb R\to\mathbb R$
such that, for every compact set $K\subset\mathbb R$, the class of
functions obtained from $g$ by repeated addition and composition is
dense in $C(K)$.
\end{corollary}

\begin{corollary}\label{cor:single_function_multivariate}
There exists a continuous non-affine function
$g:\mathbb R\to\mathbb R$
such that, for every compact set $K\subset\mathbb R^n$, the class
generated from $g$ and the coordinate functions
$x_1,\dots,x_n$
by repeated addition and composition is dense in $C(K)$.
\end{corollary}

\medskip

We stress again that Theorems~\ref{thm:two_function_univariate}
and~\ref{thm:two_function_multivariate} are general results: they apply
to every continuous non-affine function when used together with the
affine generator $t\mapsto 1-t/2$. In contrast, Corollaries
\ref{cor:single_function_univariate} and
\ref{cor:single_function_multivariate} are based on a specific
continuous non-affine function constructed in
Theorem~\ref{thm:single_function}. Note that the reduction from two
generators to one is achieved through an explicit construction.

\bibliographystyle{plain}

\end{document}